\documentclass[12pt]{article}
\usepackage{bbding}
\usepackage{mathrsfs}
\usepackage{bbm}
\usepackage{fancyhdr,graphicx}
\usepackage{amsfonts}
\usepackage{amssymb}
\usepackage{amsthm}
\usepackage{newlfont}
\usepackage{xcolor}

\usepackage{ifpdf}
\usepackage{graphicx}
\usepackage{epstopdf}
\usepackage{epsfig}
\usepackage{psfrag}
\usepackage{enumitem}
\usepackage{epstopdf}
\usepackage{verbatim}
\usepackage{color}

\newtheorem{thm}{Theorem}[section]
\newtheorem{Lemma}[thm]{Lemma}

\newtheorem{defn}[thm]{Definition}

\newtheorem{con}[thm]{Conjecture}

\usepackage{latexsym, bm}
\title{{On minimal $k$-factor-critical planar graphs}\footnote{E-mail addresses: qlli@lzu.edu.cn (Q. Li), flianglu@163.com (F. Lu) and zhanghp@lzu.edu.cn (H. Zhang).
This work is supported by NSFC (Grant Nos. 12271235 and 12271229), Fujian Key Laboratory of Granular Computing and Applications (Minnan Normal University).}}
\author{Qiuli Li$^1$, Fuliang Lu$^2$\footnote{The
corresponding author.}, Heping Zhang$^1$ \\
\small{1. School of Mathematics and Statistics, Lanzhou University, Lanzhou, China} \newline \\ {\small 2. School of Mathematics and Statistics, Minnan Normal University, Zhangzhou,  China}}

\date{}
\begin{document}

\maketitle

\begin{abstract}
 A graph of order $n$ is said to be \emph{$k$-factor-critical}
 ($0\leq k <n$) if the removal of any $k$ vertices
 results in a graph with a perfect matching.
  A $k$-factor-critical graph $G$ is  \emph{minimal} if $G-e$ is not $k$-factor-critical for any edge $e$ in $G$.
 Favaron and  Shi  posed the conjecture that every minimal $k$-factor-critical graph  is of minimum degree $k+1$ in 1998. In this paper, we confirm the conjecture for planar graphs.  \vskip 0.3cm


\end{abstract}

\topmargin -8mm \textwidth 16cm \textheight 20cm \lineskip 0.2cm
\linespread{1.2}
 \section{Introduction}

All graphs considered in this paper are  simple, finite and undirected. We follow \cite{BM08} for undefined notation and terminologies.
 A \emph{perfect matching} $M$ of a graph is a set of edges such that each vertex is incident with exactly one edge of $M$.
  A \emph{factor-critical} graph is a graph
in which the removal of any vertex results in a graph with a perfect matching.  A graph with at least one edge is called \emph{bicritical} if, after  removing any pair of distinct vertices, the left graph has a perfect matching.
    Factor-critical graphs
and bicritical graphs, introduced by Gallai \cite{ga} and   Lov\'asz \cite{lo}, respectively, play a vital role in matching theory.
  As a natural generalization of  factor-critical graphs
and bicritical graphs,
  Favaron \cite{O.Favaron} and Yu \cite{yu} independently introduced $k$-factor-critical graphs:
 a graph $G$ of order $n$ is said to be \emph{$k$-factor-critical}  (\emph{$k$}-fc for short)  with $0\leq k < n$
if the removal of any $k$ vertices  results in a graph with a perfect matching. Obviously, if $G$ is a $k$-fc graph with $n$ vertices, then $k$ and $n$ are of the same parity. 
The following lemma is about the connectivity of a $k$-factor-critical graph.

\begin{Lemma}[\cite{O.Favaron}]\label{connectivity}
Let $G$ be a $k$-factor-critical graph with order $n$ and $1\leq k < n$. Then $G$ is $k$-connected and  $(k+1)$-edge-connected.
\end{Lemma}

  The degree of a vertex $v$ in a graph $G$, denoted by $d_G(v)$, is the number of edges of $G$ incident   with $v$.
   Denote by $\delta (G)$ the minimum degree of a graph  $G$.   A $k$-fc graph $G$ is called \emph{minimal} if the deletion of any edge in $G$ results in a subgraph which is not $k$-fc. 
  By Lemma \ref{connectivity}, a minimal $k$-fc graph is $k$-connected and $(k+1)$-edge-connected, which implies that $\delta(G)\geq k+1$.  Favaron and  Shi \cite{Shi} asked the question of whether every minimal $k$-fc graph has minimum degree $k + 1$;
Zhang et al. \cite{zh} formally reproposed the following conjecture.

\begin{con}[\cite{Shi,zh}]\label{Shi}
Let $G$ be a minimal $k$-factor-critical graph. Then $\delta(G)= k+1$.
\end{con}

Regarding the ear decomposition of factor-critical graphs, it is clear that the  minimum
degree of a minimal factor-critical graph is 2, which confirms the conjecture for $k=1$ \cite{Shi}. This is also a motivation that Favaron and  Shi posed the question (it can be checked that the  minimum
degree of a minimal 0-factor-critical graph is 1).  Favaron and  Shi \cite{Shi} confirmed this conjecture for three large $k$'s with $k\in\{n-6, n-4, n-2\}$.   Guo et al. \cite{Guo2,Guo3,Guo4} confirmed this conjecture for  $k\in\{2, n-10, n-8\}$, and also for claw-free graphs \cite{Guo1}. A graph is called \emph{planar} if it can be drawn on the plane such that its edges only intersect at their end points. Such a drawing is
called a \emph{plane} graph. We confirm the conjecture to be true for planar graphs in this paper.

\begin{thm}\label{minimumdegree}
Let $G$ be a minimal $k$-factor-critical planar graph. Then $\delta(G)=k+1$.
\end{thm}

We mainly pay attention to Conjecture \ref{Shi} for planar graphs. Aside this conjecture, 
a problem that one may be interested in is to determine the number of vertices with minimum degree in minimal $k$-fc graphs. It should be noted that a 3-connected 2-fc graph is called a {\em brick}, which is  a building block of matching covered graphs; see \cite{plummer} for details.
Norine and Thomas \cite{nt} conjectured that every minimal brick (minimal 3-connected 2-fc graph) $G$ has at least $\alpha |V (G)|$ ($\alpha >0$) vertices of degree 3, where a brick $G$ is \emph{minimal} if $G-e$ is not a brick for any edge $e$ in $G$. Along this line, what can we say about the number of vertices with degree 4 in  minimal 4-connected 3-fc graphs?
Thomas and Yu \cite{ty} showed that  every graph which can be obtained from a 4-connected planar graph by deleting two vertices is Hamiltonian. Then every 4-connected planar graph of odd order is 3-fc.
 Therefore, if $G$ is a 4-connected minimal 3-fc planar graph, then $G-e$ is not $4$-connected for any edge $e$ in $G$. That is to say, a 4-connected minimal 3-fc planar graph is minimal 4-connected, where  a 4-connected graph is {\em minimal} if the deletion
of each edge results in a non-4-connected graph.
For a minimal $k$-connected graph $G$, Mader showed that $\delta(G) = k$ \cite{ma}, and
further obtained that $G$ has at least $\frac{(k-1)|V(G)|+2}{
2k-1}$ vertices of degree $k$ \cite{ma2}. By Mader's result, every 4-connected minimal 3-fc planar graph has at least $\frac{3|V(G)|+2}{
7}$ vertices of degree $4$.

\section{Preliminaries}

We begin with some notation. Let $G$ be a graph with vertex set $V(G)$ and edge set $E(G)$.  For $X,Y\subseteq V(G)$, by $E_G(X,Y)$  we mean the set of edges of $G$ with one end point in $X$ and the other end point in $Y$; by $G[X]$ we mean the subgraph of $G$ induced by $X$.  Denote by $N_G(X)$, or simply $N_G(u)$ when $X=\{u\}$,  the set of all the vertices in $\overline{X}=V(G)\setminus X$ adjacent to some vertex in $X$. Let $H$ be a subgraph of $G$. We usually simply write $N_G(V(H))$ as $N_G(H)$. Let  $|N_G(H)|=n_G(H)$. If no confuse occurs, the subscript will be omitted.
Assume that $X$ is a vertex cut, that is, $G-X$ has more components than $G$. We say
 $X$  \emph{separates} $H'$ if  $ H'$ is a  component or a union of  components  of $G-X$ such that $V(G)\setminus (X\cup  V(H'))\neq \emptyset$.

We first present several results related to planar graphs, such as upper bound of the number of edges in a bipartite planar graph, upper bound of the minimum degree of a planar graph, and the classic characterization theorem for planar graphs. As we will see, these will play crucial roles in our proof.

We  recall the well-known \emph{Euler's
Formula} on plane graphs, that is,
if we denote $n,m$ and $f$  the numbers of vertices, edges and faces of a
connected plane graph respectively, then $n-m+f=2$. By considering the relationships between the sum of  degrees of vertices (faces) and the number of edges, one can easily show that any planar graph has a vertex of small degree, see Corollary 10.22 in \cite{BM08} for example.

\begin{Lemma}\label{degree5}
Let $G$ be a planar graph. Then $\delta(G)\leq 5$.
\end{Lemma}

The next lemma gives an estimation of the number of edges in a
bipartite planar graph, which is also obtained by the Euler's
formula of graphs on  plane graphs and the property that the
length of each face in a bipartite plane graph is at least four.
\begin{Lemma}\label{bipartiteontheplane}
Let $G$ be a bipartite planar graph  with $n$ vertices
and $m$ edges.
 Then $m\leq 2n-4$.
\end{Lemma}

Let $H$ be a graph. A graph $T$ is called
an $MH$ if its vertex set admits a partition $\{V_{x} | x \in V (H) \}$ such that  $T[V_{x}]$ is connected, and distinct vertices $x, y \in V(H)$ are adjacent in $H$
if and only if $E_T(V_{x},V_{y})\neq \emptyset$. 
If a graph $G$ contains an $MH$ as a subgraph, then $H$ is a \emph{minor} of $G$.
The following classic theorem, named Kuratowski's theorem, gives a characterization of planar graphs.

\begin{thm}[\cite{Kuratowski, Wagner}]\label{planar}
A graph $G$ is planar if and only if it contains neither $K_{5}$ nor $K_{3,3}$ as a minor.
\end{thm}

We second present several useful properties on matchings.
The first one is the famous Tutte's Perfect Matching (1-Factor)  Theorem.  We say a component is \emph{odd} (or \emph{even}) if it is of odd (or even) order. For $S\subseteq V(G)$, let us denote by
 $o(G-S)$ the number of odd components of $G-S$.

\begin{thm}[\cite{tut}] \label{Tutte}
A graph $G$ has a perfect matching if and only if $o(G-S)\leq |S|$ for any $S\subseteq V(G)$.
\end{thm}

 The following result can be directly obtained by the definition of 3-factor-criticality. As it will be used  time and time again, we state it as a lemma.
 \begin{Lemma}\label{3cuteven}
Let $G$ be a 3-factor-critical graph and $S\subseteq V(G)$. If $|S|=3$, then every component of $G-S$ is even.
\end{Lemma}

 At the end of this section, we give the following result that will play an important role  in the sequel.
 \begin{Lemma}\label{ngeq4}
Let $G$ be a 3-factor-critical graph of order at least 4 and $H$ be a subgraph of $G$. If $|V(H)|$ is odd and $N(H)$ is a vertex cut, then $n(H)\geq 4$.
\end{Lemma}
\begin{proof}
 Suppose, to the contrary, that $n(H)\leq 3$. Since $G$ is 3-fc and $|V(G)|\geq 4$, $G$ is of odd order and further  3-connected  by Lemma \ref{connectivity}. As $N(H)$ is a vertex cut, $n(H)=3$ by the 3-connectedness of $G$. By Lemma \ref{3cuteven}, $|V(H)|$ is even, a contradiction. Therefore, $n(H)\geq 4$.
\end{proof}

\section{Proof of Theorem \ref{minimumdegree}}

 We first prove that Conjecture \ref{Shi} is true for minimal 3-fc planar graphs, that is, Theorem \ref{minimumdegree} holds for $k=3$. It is relatively easy to handle for the other $k$'s.

\begin{defn}Let $uv\in E(G)$.
We say  a vertex cut $X_{u}(uv)\subseteq V(G)$ \emph{(}resp.  $X_{v}(vu)\subseteq V(G))$ satisfies  Property $P$ if

{\rm{(i)}} $|X_{u}(uv)|=4$ \emph{(}resp.  $|X_{v}(vu)|=4)$;

{\rm{(ii)}} $v\in X_{u}(uv)$ \emph{(}resp. $u\in X_{v}(vu))$;

{\rm{(iii)}} $u$ lies in an odd component, say $G_{u}$, of $G-X_{u}(uv)$ $(v$ lies in an odd component, say $G_{v}$, of $G-X_{v}(vu))$; and

{\rm{(iv)}} $E(V(G_u), \{v\})=\{uv\}$ \emph{(}resp. $E(V(G_v), \{u\})=\{uv\})$.\\
If no confuse
occurs, one can simply write $X_u$ for $X_u(uv)$ \emph{(}resp. $X_{v}$ for $X_{v}(vu))$.
\end{defn}

\begin{Lemma} \label{propertyP}
Let $G$ be a minimal 3-factor-critical planar graph. If $\delta(G)\geq 5$, then for any edge $e=uv\in E(G)$, there exists a vertex cut $X_{u}$  and a vertex cut $X_{v}$ of $G$ satisfying Property $P$.
\end{Lemma}

\begin{proof}
Since $G$ is minimal 3-fc, $G-e$ is not 3-fc. By the definition of a 3-fc graph, there exists an $S'\subseteq
V(G-e)$ of size $3$ satisfying that $G-e- S'$ has no perfect
matchings. Set $G'=G-e- S'$. Then, by Theorem \ref{Tutte},
$G'$ contains an $S''\subseteq V(G')$ such that $
o(G'-S'')\geq|S''|+1$.
  Moreover, since $G'$ is of even order ($|V(G')|=|V(G)|-|S'|=|V(G)|-3$),
  $|S''|$ and $o(G'-S'')$ have the same parity, and it follows that $o(G'-S'')\geq |S''|+2$. On the other hand, since $G$ is 3-fc, $G'+e=G-S'$ has a perfect matching. By Theorem \ref{Tutte}, for $S''$, $o(G'+e-S'')\leq |S''|$ holds. As adding  the edge $e=uv$ to $G'-S''$ can only destroy at most two odd components of it, we have $o(G'+e-S'')\geq o(G'-S'')-2$. Combining the above inequalities, we have $|S''|\geq o(G'+e-S'')\geq o(G'-S'')-2\geq |S''|$. It follows that $G'-S''$ has exactly $|S''|+2$ odd components (that is $o(G'-S'')=|S''|+2$),  denoted by $G_{i}$ ($1\leq i
  \leq |S''|+2$),  and $u,v$ lie in two different $G_{i}$'s.  For convenience,  we denote the components containing $u$ and $v$  by $G_{u}$ and $G_{v}$ respectively.

  Now we are to show that each $G_{i}$ ($1\leq i
  \leq |S''|+2$) is connected to at least 4 vertices in $\overline{V(G_{i})}$, that is, $n_G(G_{i})\geq 4$.   If  $|S''|=0$, then $G-e- S'-S''$ has exactly two odd components: $G_{u}$ and $G_{v}$. Since $\delta(G)\geq 5$, $d_{G}(u)\geq 5$ and  $d_{G}(v)\geq 5$. As   $|S'\cup S''|=|S'|+|S''|=3$ and   $E_G(V(G_{u}), V(G_{v}))=\{uv\}$, we have that $|V(G_{u})|\geq 2$ and $|V(G_{v})|\geq 2$ (indeed, $|V(G_{u})|\geq 3$ and $|V(G_{v})|\geq 3$ as both $G_{u}$ and $G_{v}$ are odd components). It implies that $\overline{V(G_{i})\cup N_G(G_{i})}\neq \emptyset$.  If  $|S''|\geq 1$, then there are at least $|S''|+2\geq 3$ odd components in $G-e- S'-S''$; this means that for each $i$, $\overline{V(G_{i})\cup N_G(G_{i})}\neq \emptyset$.
   By the above argument, we conclude  that $N_G(G_{i})$ is a vertex cut.  Therefore,  $n_G(G_{i})\geq 4$ by Lemma \ref{ngeq4}.

\begin{figure}[htbp]
\begin{center}
\includegraphics[totalheight=4.5 cm]{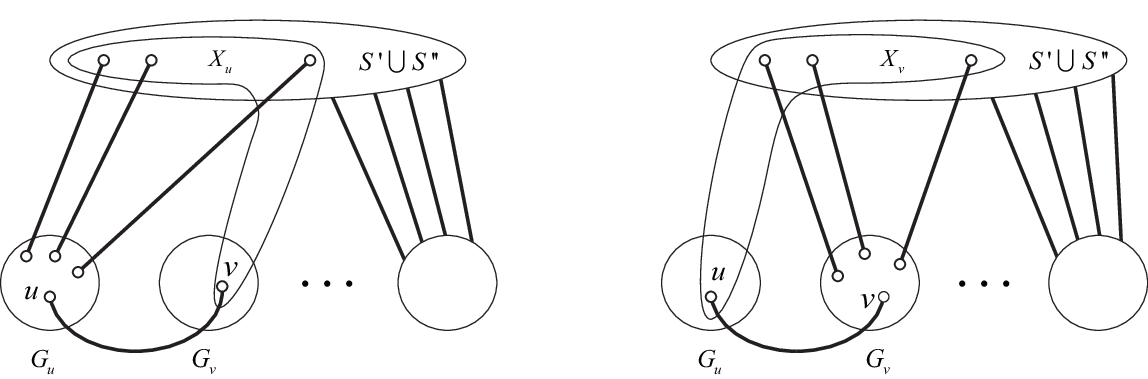}
\caption{\label{5}\small{Illustration for the proof of  Lemma \ref{propertyP}.}}
\end{center}
\end{figure}

We may assume that $G$ is a plane graph. In $G-e=G-uv$, we delete the even components of $G'-S''$, contract each component $G_{i}$ ($1\leq i
\leq |S''|+2$) into a singleton,
  and delete the edges in $G[S'\cup S'']$ and the multiple edges produced by the contraction. Denote the
  resulting  graph by $H$.  We can easily  see that  $H$ is a plane graph too. Let $B$  consist of the singletons gotten by contracting each component  $G_{i}$ ($1\leq i\leq |S''|+2$). Then $H$ is  a bipartite graph with partite sets $S'\cup S''$ and  $B$. Denote by $u'$ and $v'$ the vertices of  $H$ gotten by contracting $G_{u}$ and $G_{v}$
  into  singletons respectively. Then $\{u', v'\}\subseteq B$.
  Recalling that $n_G(G_{i})\geq 4$, we have  $d_H(x)\geq 4$ for $x\in B\setminus\{ u',v'\}$; $d_H(u')\geq 3$, $d_H(v')\geq 3$ as $uv\in E(G)$.
   We now estimate $|E(H)|$ in two different ways. On one hand, we count the number of edges sending out from $B$ and obtain that $|E(H)|\geq 4(|B|-2)+3\times 2=4(o(G'-S'')-2)+3\times 2=4|S''|+6$. On the other hand, by Lemma \ref{bipartiteontheplane}, $|E(H)|\leq 2(|S'\cup S''|+|B|)-4= 2(|S'|+|S''|+|S''|+2)-4=4|S''|+6$  (as $|S'|=3$). Therefore, $|E(H)|=4|S''|+6$ and further,  $d_H(x)= 4$ for $x\in B\setminus \{ u',v'\}$; $d_H(u')=d_H(v')= 3$. Then, in $G$,
   each  $G_{i}$ has exactly 4 neighbors in $\overline{V(G_{i})}$ for $1\leq i
  \leq |S''|+2$. In particular,  $n_G(G_{u})=4$ and $n_G(G_{v})=4$ (recalling $uv\in E(G)$).
  By the above argument,   $N_G(G_{u})$ is our desired $X_{u}$ and $N_G(G_{v})$ is our desired $X_{v}$, see Figure \ref{5} for illustration. The lemma follows.
   \end{proof}


 \begin{thm}\label{3fc}
Let $G$ be a minimal 3-factor-critical planar graph. Then $\delta(G)=4$.
\end{thm}

\begin{proof} Since $G$ is  3-fc, $G$ is 4-edge-connected by Lemma \ref{connectivity}, implying that $\delta(G)\geq 4$. To prove that $\delta(G)=4$, we only need to show that $\delta(G)\leq 4$. Suppose to the contrary that $\delta(G)\geq 5$. Then by Lemma \ref{propertyP}, for any edge $e=u'v'\in E(G)$, there exist two  4-vertex cuts (vertex cuts of size 4) $X_{u'}$  and  $X_{v'}$ of $G$ both satisfying Property $P$.  Since $G$ is of odd order (as $G$ is 3-fc), each 4-vertex cut separates at least one odd component. Let $X$ be a 4-vertex cut which separates an odd component of smallest order among all 4-vertex cuts of $G$ and we denote such one odd component by $O$.
Since $\delta(G)\geq 5$, $|V(O)|\neq 1$, and further $E(O)\neq \emptyset$.  Let   $uv\in E(O)$. By Lemma \ref{propertyP}, we have a 4-vertex cut $X_{u}$ satisfying Property $P$. Now we have two different 4-vertex cuts $X$ and $X_{u}$  as $v\in X_{u}$ but $v\notin X$.  So we have two different partitions of $V(G)$ regarding $X$ and $X_{u}$  respectively:   $\{V(O), X, \overline{V(O)\cup X}\}$ and   $\{ V(G_{u}), X_{u}, \overline{V(G_{u})\cup  X_{u}}\}$, where $G_u$ is the odd component of $G-X_u$ containing $u$.
  Note that  both $|\overline{V(O)\cup X}|$ and $|\overline{V(G_{u})\cup  X_{u}}|$  are even, as both $|V(O)|$ and $|V(G_{u})|$ are odd, $|X_{u}|=|X|=4$ and $|V(G)|$ is odd.


We refine the two partitions each other as the following possible nine subsets (see Figure \ref{fig3} (left)): $V_{1}=V(G_{u})\cap V(O)$, $V_{2}=V(G_{u})\cap \overline{V(O)\cup X}$, $V_{3}=\overline{V(G_{u}) \cup X_{u}} \cap \overline{V(O)\cup X}$, $V_{4}=\overline{V(G_{u}) \cup X_{u}} \cap V(O)$, $X_{1}=X\cap V(G_{u})$, $X_{2}=X\cap \overline{V(G_{u}) \cup X_{u}}$, $X_{u}^1=X_{u}\cap V(O)$, $X_{u}^2=X_{u}\cap \overline{V(O)\cup X}$ and $C=X\cap X_{u}$. As $v\in X_{u}$ and $v\in V(O)$, one can see that $v\in X_{u}^1$. As $u\in V(G_{u})$ and $u\in V(O)$, we have $u\in V_{1}$.  Let $x_{1}=|X_{1}|$, $x_{2}=|X_{2}|$, $y_{1}=|X_{u}^1|$, $y_{2}=|X_{u}^2|$ and $c=|C|$. Please see Figure \ref{fig3} (right) for illustration. Then $|X|=x_{1}+x_{2}+c=|X_{u}|=y_{1}+y_{2}+c=4$, and
\begin{equation}\label{3.1}
x_{1}+x_{2}+y_{1}+y_{2}+2c=8.
\end{equation}

 For convenience, let $S_{1}=X_{1}\cup C \cup X_{u}^1$, $S_{2}=X_{1}\cup C \cup X_{u}^2$, $S_{3}=X_{u}^2\cup C \cup X_{2}$ and $S_{4}=X_{2}\cup C \cup X_{u}^1$. Since $V_{1}\neq \emptyset$  ($u\in V_{1}$) and $\emptyset \neq \overline{V(O)\cup X} \subseteq \overline{ V_1\cup S_1}$, $S_{1}$ is a vertex cut of $G$. Hence $|S_{1}|=x_{1}+c+y_{1}\geq 3$ (as $G$ is 3-connected). The next claim shows that $|S_{1}|$ is at least 5.
\begin{figure}[htbp]
\begin{center}
\includegraphics[totalheight=5 cm]{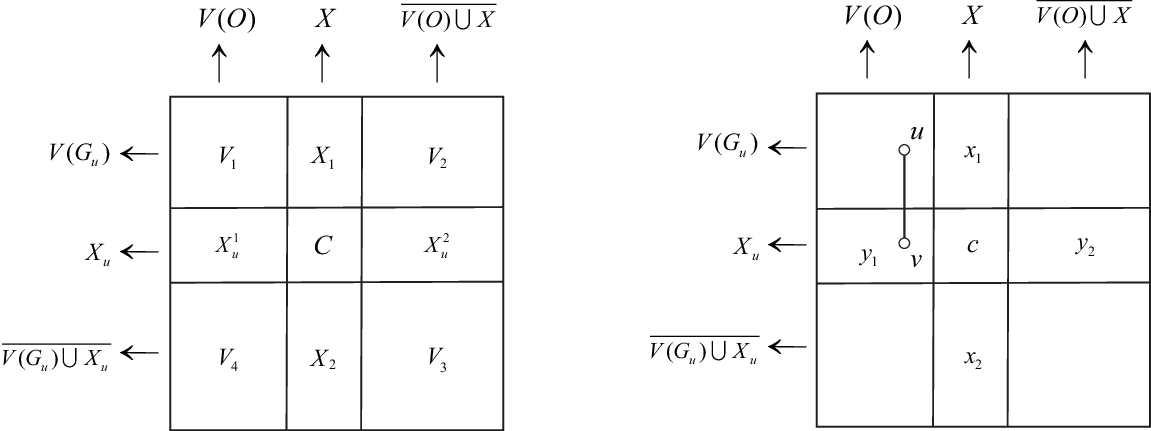}
\caption{\label{fig3}\small{Illustration for a  partition of $V(G)$ (the rectangles in the figure mean the vertex sets).}}
\end{center}
\end{figure}

\vspace{.2cm}

{\bf{Claim 1.}} $|S_{1}|=x_{1}+c+y_{1}\geq 5$.

\vspace{.2cm}

\emph{Proof.} Suppose, to the contrary, that $|S_{1}|\leq 4$. Then $|S_{1}|=3$ or 4. If $|S_{1}|=3$, then $|V_{1}|$ is even by Lemma \ref{3cuteven}, as $S_{1}$ is a vertex cut. 
Consequently, $S_{1}\cup \{u\}$ is a 4-vertex cut of $G$  separating $G[V_{1}\setminus \{u\}]$. As $|V_{1}\setminus \{u\}|$ is odd, $G[V_{1}\setminus \{u\}]$ contains an odd component  (maybe itself if it is connected), say $H$. Then $S_{1}\cup \{u\}$ separates $H$. As $ V(H)\subseteq V(O)\setminus \{u, v\}$ and $\{u,v\} \subseteq V(O)$,  $|V(H)|<|V(O)|$ holds. That is to say, we obtain a  4-vertex cut  separating an  odd component with smaller order than $O$, which contradicts the choice of $X$.  If $|S_{1}|=4$, then by the choice of $X$  and $|V_{1}|<|V(O)|$, $|V_{1}|$ must be even. In such case, as $X_{u}$ is a 4-vertex cut satisfying Property $P$, we have $ E(V(G_u), \{v\})=\{uv\}$, which implies that  $(S_{1}\cup\{u\})\setminus\{v\}$ is a 4-vertex cut separating $G[V_1\setminus\{u\}]$. Noting that $|V_1\setminus\{u\}|$ is odd and $ V_1\setminus\{u\} \subsetneq V(O)$,
$G[V_1\setminus\{u\}]$ contains an odd component of  $G-((S_{1}\cup \{u\})\setminus \{v\})$ which is of smaller order than $O$, contradicting the choice of $X$ again. Claim 1 is proved.

\vspace{.2cm}
The following claim can be easily obtained by the 3-connectedness of $G$, which  will be used frequently.

\vspace{.2cm}
{\bf{Claim 2.}} For $i=2,3,4$, if $|S_{i}|\leq 2$, then $V_{i}$ is an empty set.

\vspace{.2cm}

We now show that $|V_{3}|$ must be even.

\vspace{.2cm}
{\bf{Claim 3.}} $|V_{3}|$ is even.

\vspace{.2cm}

\emph{Proof.}  Since  $|S_{1}|=x_{1}+y_{1}+c\geq 5$ by Claim 1, $|S_{3}|=x_{2}+y_{2}+c\leq 3$ by Equation  (\ref{3.1}).
If $V_{3}$ is  an empty set, then, obviously, $|V_{3}|$ is even. So we assume $V_{3}\neq \emptyset$. Then $|S_{3}|=3$ by Claim 2.
 By Lemma \ref{3cuteven}, we have  $|V_{3}|$ is even as $S_{3}$ is a vertex cut. So the claim follows.

\vspace{.2cm}
{\bf{Claim 4.}} $X_{2}\neq \emptyset$. That is, $x_{2}=|X_{2}| \neq 0$.

\vspace{.2cm}

\emph{Proof.} Suppose, to the contrary, that $X_{2}=\emptyset$. By Claim 3, $|V_{3}|$ is even.  Recalling that $|\overline{V(G_{u})\cup X_{u}}|$ is even, $\overline{V(G_{u})\cup X_{u}}=V_{3}\cup X_{2} \cup V_{4}$ and  $|V_{3}|$ is even, we have that $|V_{4}|$ is even. Consequently, $|V_{4}\cup \{v\}|$ is odd.  Since $E(V(G_{u}), \{v\})=\{uv\}$, we have that $(X_{u}\cup \{u\})\setminus \{v\}$ is a 4-vertex cut (as $d_G(u)>4$, $|V(G_{u})|>1$), and separates an odd component contained in $G[V_{4}\cup \{v\}]$ (maybe itself) which is of smaller order than $O$ (note that ($V_{4}\cup \{v\}) \subsetneq V(O)$), contradicting the choice of $X$. The proof is complete.

\vspace{.2cm}

In the sequel, we will find the contradiction by distinguishing the following cases according to the value of $c$.

\vspace{.2cm}
{\bf{Case 1.} $c=0$}.
\vspace{.2cm}

In this case, we have $x_{1}+y_{1}\geq 5$ by Claim 1 and further $x_{2}+y_{2}\leq 3$ by Equation  (\ref{3.1}). By Claims 3 and 4,  $|V_{3}|$ is even and $|X_{2}|=x_{2}\geq 1$. Therefore, $1\leq x_{2}\leq x_{2}+y_{2}\leq 3$.

If $x_{2}=1$, then $y_{2}\leq 2$. We assume that $y_{2}\leq 1$ firstly.  Then $|S_{3}|=x_{2}+y_{2}\leq 2$. By Claim 2, $V_{3}$ is an empty set. Then $V_{4}\cup X_{2}=\overline{V(G_{u})\cup X_{u}}$. Therefore,  $(X_{u}\cup \{u\})\setminus \{v\}$ is a 4-vertex cut separating an odd component $H$ in $G[V_{4}\cup X_{2} \cup \{v\}]$ (maybe itself).  As $c=0$, we have  $y_{1}+y_{2}=4$. It follows that $y_{1}\geq 3$. Hence  $|X_{2}|=1<y_1= |X_{u}^1| $. Then $|V(H)|\leq |V_{4}\cup X_{2}\cup \{v\}|= |V_{4}|+ |X_{2}|+1 < |V_{4}|+ |X_{u}^1|+|V_1|=|V(O)|$,
 contradicting the choice of $X$.
%
Now we consider the case that $y_{2}=2$. Then  $y_1=2$. So $|S_{3}|=|S_{4}|=3$. If $V_{3}=\emptyset$ then $|V_{3}|$ is even; if $V_{3}\neq \emptyset$, then $|V_{3}|$ is also  even  by Lemma \ref{3cuteven} as  $S_{3}$ is a  vertex cut. Similarly, $V_{4}$ is of even size.
 Consequently, $|\overline{V(G_{u})\cup X_{u}}|=|V_{3}\cup V_{4}\cup X_{2}|=|V_{3}|+|V_{4}|+1$ is odd, a contradiction.

If $x_{2}=2$, then $x_{1}=2$ and $y_{2}\leq 1$.  If $y_{2}=0$,  then $y_1=4$ and $|S_{3}|=2$.
By Claim 2, $V_{3}$ is an empty set,   and further $(X_{u}\cup \{u\} )\setminus \{v\}$ is a desired 4-vertex cut, contradicting the choice of $X$ (similar to the case   that $x_{2}=1$ and $y_{2}\leq 1$).  If $y_{2}=1$, then $|S_{2}|=|S_{3}|=3$.
If $V_{2}\neq \emptyset$, then $S_{2}$ is a  vertex cut, implying that $|V_{2}|$ is  even by Lemma \ref{3cuteven}.
If $V_{2}=\emptyset$ then $|V_{2}|$ is also even. Similarly, $V_{3}$ is of even size.
It implies that  $|\overline{V(O)\cup X}|=|V_{2}|+|V_{3}|+y_{2}=|V_{2}|+|V_{3}|+1$ is odd, a contradiction.

If $x_{2}=3$, then  $x_{1}=1$, $y_{2}=0$ and $y_{1}=4$.  In this case, we have that $V_{2}$ is an empty set. So
 $ |V(G_{u})|=|V_1|+|X_1|=|V_1|+1<|V_1|+|X_u^1|+|V_4|=|V(O)| $.
But $X_{u}$ is a
 4-vertex cut separating the odd component  $G_{u}$,
  contradicting the choice of $X$.

\vspace{.2cm}
{\bf{Case 2.} $c=1$}.

\vspace{.2cm}

In this case, we have $x_{1}+y_{1}\geq 4$ by Claim 1 and further $x_{2}+y_{2}\leq 2$ by Equation  (\ref{3.1}). We also have $|X_{2}|=x_{2}\geq 1$ by Claim 4. Therefore, $1\leq x_{2}\leq x_{2}+y_{2}\leq 2$.


If $x_{2}=1$, then $y_{2}\leq 1$.  If  $y_{2}=0$, then $y_{1}= 3$ and $|S_{3}|=x_{2}+c+y_{2}= 2$, which implies  that $|V_{3}|=0$ by Claim 2. Then $V_{4}\cup X_{2}=\overline{V(G_{u})\cup X_{u}}$ is of even size and  $(X_{u}\cup \{u\}) \setminus \{v\}$ is a  4-vertex cut separating an odd component $H$ in $G[V_{4}\cup X_{2} \cup \{v\}]$ (maybe itself). Since
 $|V(H)|\leq |V_{4}\cup X_{2}\cup \{v\}|=|V_{4}|+2$ and $|V(O)|=|V_{4}\cup X_{u}^{1}\cup V_{1}|=|V_{4}|+y_{1}+|V_{1}|\geq |V_{4}|+4$, we have $|V(H)|<|V(O)|$. Then $(X_{u}\cup \{u\}) \setminus \{v\}$ is a  4-vertex cut contradicting the choice of $X$.  Now we are left to the case   that $y_{2}=1$.  By Claim 3, $|V_{3}|$ is even. It implies that $|V_{4}|$  is odd (recall that $|V_4\cup V_3\cup X_2|$ is even). Note that $|S_{4}|=|X_u^1\cup C\cup X_2|=4$. Then $S_{4}$ is a desired 4-vertex cut contradicting the choice of $X$.

If $x_{2}=2$, then  $y_{2}=0$. In this case, we can immediately deduce that $|S_{2}|=2$ and  $V_{2}$ is an empty set by Claim 2. Noting that $x_1=1$ and $y_1= 3$,  we have $|V(G_{u})|=|V_{1}|+|X_{1}|=|V_{1}|+1$ and $|V(O)|=|V_{1}|+|X_{u}^{1}|+|V_{4}|\geq |V_{1}|+3$. It follows that $|V(G_{u})|<|V(O)|$ and  $X_{u}$ is a desired 4-vertex cut contradicting the choice of $X$.

\vspace{.2cm}
{\bf{Case 3.} $c=3$}.
\vspace{.2cm}

By Claim 1, $x_{1}+y_{1}\geq 2$.  By Equation  (\ref{3.1}), $x_{2}+y_{2}\leq 0$. Then $x_{2}=y_{2}=0$, which contradicts that $x_{2}\neq 0$ by Claim 4.


\vspace{.2cm}

{\bf{Case 4.} $c=2.$}

\vspace{.2cm}

By Claim 4, $|X_{2}|=x_{2}\geq 1$. By Claim 1, we have $x_{1}+y_{1}\geq 3$.  Consequently, $x_{2}+y_{2}\leq 1$ by Equation  (\ref{3.1}).  Therefore, $x_{2}=1$, $y_{2}=0$, and further $x_{1}=1, y_{1}=2$.

We claim that $V_{2}\neq \emptyset$. Otherwise, $V_{2}= \emptyset$. Then
$ |V(G_u)|<|V(O)| $.
So  $X_{u}$ is a desired 4-vertex cut
separating the odd component $G_u$  of order smaller than $O$,
 contradicting the choice of $X$.
 Let $B_{2}$ be any component of $G[V_{2}]$. Since $|S_{2}|=x_{1}+y_{2}+c=3$,  $N(B_{2})=S_{2}$ as $G$ is 3-connected.  Moreover,  $|V_{2}|$ is even by Lemma \ref{3cuteven} (as $S_{2}$ is a   vertex cut). Further, since $|V(G_{u})|=|V_{1}|+|V_{2}|+1$ is odd, $|V_{1}|$ is even. Hence $|V_{1}\setminus \{u\}|$ is odd and $G[V_{1}\setminus \{u\}]$ contains at least one odd component. Let $B_{1}$ be  such an odd component. Then $N(B_{1})\subseteq  S_{1}\cup \{ u\}$. As $S_{1}$ is a vertex cut, $N(B_{1})$ is a vertex cut. It implies that $n(B_{1})\geq 3$ by the 3-connectedness of $G$. Besides, since $|V(B_{1})|$ is odd,  $n(B_{1})\neq 3$ by Lemma \ref{3cuteven}. In a  word, $n(B_{1})\geq 4$.
 As $V(B_{1}) \subsetneq V(O)$, $n(B_{1})\neq 4$ by the choice of $X$. So $n(B_{1})\geq 5$. Recalling that $ E(V(G_u), \{v\})=\{uv\}$,  $N(B_{1})=(S_{1}\cup\{u\})\setminus\{v\}$ and $n(B_{1})=5$. Now we turn our attention to $G[V_4]$.  Since $|V(O)|$ is odd, $|V(O)|=|V_{1}|+2+|V_{4}|$ and $|V_{1}|$ is even, one can conclude that $|V_{4}|$ is odd. Let $B_{4}$ be an odd component of $G[V_4]$. By a complete similar argument as we have done for $B_{1}$, we obtain that $|N(B_{4})|=5$ and $N(B_{4})=S_{4}$. We also claim that $V_{3}\neq \emptyset$. Otherwise, $(X_{u}\cup \{u\})\setminus \{v\}$ is a 4-vertex cut separating $G[V_4\cup X_2\cup \{v\}]$.
 As $|V_{1}|$ is even and not zero, $|V_4\cup X_2\cup \{v\}| =|V_4|+2<|V_4\cup X_u^1|+|V_1|=|V(O)|$ (note that $|X_{u}^{1}|=2$).
  Thus the 4-vertex cut $(X_{u}\cup \{u\})\setminus \{v\}$ separates a smaller odd component than $O$,
  contradicting the choice of $X$. Similar to the components in $G[V_{2}]$, for any component  $B_{3}$  in $G[V_{3}]$, we have $N(B_{3})=S_{3}$.

Let $X_{1}=\{a\}, C=\{b,c\}$ and $X_{2}=\{d\}$. If $ua\in E(G)$,
let $U_{1}=B_{1}$, $U_{2}=B_{2}\cup \{a\}$,
 $U_{3}=B_{4}\cup \{v\}$. Then $G[U_i]$ is connected for $i=1,2,3$. Moreover, for every $x\in \{u,b,c\}$, $N(x)\cap U_i\neq \emptyset$. Therefore, The subgraph induced by $U_{1}\cup U_{2}\cup U_{3}\cup \{ u,b,c\} $  contains a $K_{3,3}$ minor of $G$, contradicting that $G$ is planar by Theorem \ref{planar} (see Figure \ref{fig4}).

\begin{figure}[htbp]
\begin{center}
\includegraphics[totalheight=5 cm]{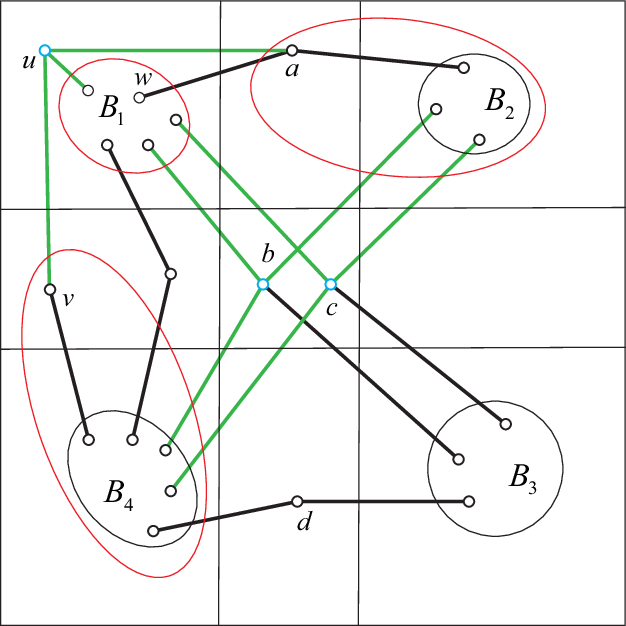}
\caption{\label{fig4}\small{The illustration for a $K_{3,3}$ minor.}}
\end{center}
\end{figure}

Next we assume that $ua\notin E(G)$. By the above argument, we know $N(B_{1})=(S_{1}\cup \{u\})\setminus \{v\}$. So there exists a vertex $w\in V(B_1)$ satisfying  $wa\in E(G)$;  see Figure \ref{fig4}.  As $w\in V(O)$, $O$ is an odd component of $G-X$ and $|V(O)|\geq 3$, there exists a vertex $z$ in $V(O)$ such that $wz\in E(O)$. Obviously, $z\neq v$ as $wv\notin E(G)$ by $E(V(G_u), \{v\})=\{uv\}$ and $w\in V(G_{u})$.  By Lemma \ref{propertyP}, there is a vertex cut $X_{w}$ satisfying Property $P$.  We do the same argument on the edge $wz$  as we have done on $uv$, using similar notation: $V_{1}'=V(G_{w})\cap V(O)$, $V_{2}'=V(G_{w})\cap \overline{V(O)\cup X}$, $V_{3}'=\overline{V(G_{w}) \cup X_{w}} \cap \overline{V(O)\cup X}$, $V_{4}'=\overline{V(G_{w}) \cup X_{w}} \cap V(O)$, $X_{1}'=X\cap V(G_{w})$, $X_{2}'=X\cap \overline{V(G_{w}) \cup X_{w}}$, $X_{w}^{1}=X_{w}\cap V(O)$, $X_{w}^{2}=X_{w}\cap \overline{V(O)\cup X}$, $C'=X\cap X_{w}$; and $S_{i}'~(i=1,2,3,4)$ is defined similarly, too.  Then ultimately,  we only need to consider the situation when all  the following statements hold.

 (i) $|C'|=2$, $|X_{1}'|=|X_{2}'|=1$,  $|X_{w}^{1}|=2$ and $ |X_w^2|=0$;

(ii) $w\in V_{1}'$, $|V_{1}'\setminus \{w\}|$ is odd and for any odd component $B_{1}'$ in $G[V_{1}'\setminus \{w\}]$, $N(B_{1}')=(S_{1}'\cup\{w\})\setminus\{z\}$;

(iii) $V_{2}'\neq \emptyset$ and for any component $B_{2}'$ in $G[V_{2}']$, $N(B_{2}')=S_{2}'$;

(iv) $V_{3}'\neq \emptyset$ and for any component $B_{3}'$ in $G[V_{3}']$, $N(B_{3}')=S_{3}'$;

(v) $|V_{4}'|$ is odd and for any odd component $B_{4}'$ in $G[V_{4}']$, $N(B_{4}')=S_{4}'$.

As $X$ is fixed, we have $\overline{V(O)\cup X}=V_{2}\cup V_{3}=V_{2}'\cup V_{3}'$.
 Recall that   $N(V_{2})=S_{2}=C\cup X_1$, $N(V_{3})= S_{3}= C\cup X_2$,
  $N(V_{2})\cap X_2= \emptyset$ and $N(V_{3})\cap X_1= \emptyset$.
  Similarly,   $N(V'_{2})=  S_{2}'= C'\cup X'_1$, $N(V'_{3})=  S_{3}'=C'\cup X'_2$,
  $N(V'_{2})\cap X'_2= \emptyset$ and $N(V'_{3})\cap X'_1= \emptyset$. So every vertex in $C$ is adjacent to some vertex in each component of $G[ V_{2}\cup V_{3}]$. Let $ T=\{t\in X: \mbox{$t$ is adjacent to some vertex in each component of $G[ V_{2}\cup V_{3}]$}\}$. As $N(V_{2})\cap X_2= \emptyset$ and $N(V_{3})\cap X_1= \emptyset$,  $T=C$.
   Similarly, $T=C'$
  (Noting the components of $G[ V_{2}\cup V_{3}]$ are also the ones of $G[ V'_{2}\cup V'_{3}]$ as $X_{u}^2=X_{w}^{2}=\emptyset$).
 Therefore, $C=C'=\{b,c\}$. Besides, because $wa\in E(G)$,  $w\in V_{1}\cap V_{1}'$ and $N(V_{1})\cap X_{2}=\emptyset$,  we have $X_{1}'=X_{1}=\{a\}$. Since $wa\in E(G)$, similar to the case that $ua\in E(G)$, we can show that $G$  contains  $K_{3,3}$ as a minor, a contradiction. In a word, if $\delta(G)\geq 5$, then we will obtain a contradiction, that is, $\delta(G)\leq 4$. The result that $\delta(G)=4$ follows.
\end{proof}

Now we are ready to prove our main result.

\vspace{.2cm}

\noindent {\bf Proof of Theorem \ref{minimumdegree}.}   As we know, if $G$ is a minimal $k$-fc graph,  then $\delta(G)\geq k+1$.  Since  $G$ is a planar graph, $\delta(G)\leq 5$ by Lemma \ref{degree5}. It follows that if $G$ is a  minimal $k$-fc planar graph, then $k\leq 4$ and the theorem holds for $k=4$ immediately. For $k=1,2$, Conjecture \ref{Shi} has been confirmed (see Introduction of this paper), and surely holds for planar graphs. By Theorem \ref{3fc}, Theorem \ref{minimumdegree} holds for $k=3$. The proof is complete. \hfill$\Box$

\end{document}